\documentclass{article}
\usepackage{amsfonts}
\usepackage{amsmath}

\newtheorem{theorem}{Theorem}

\newtheorem{corollary}[theorem]{Corollary}

\newtheorem{definition}[theorem]{Definition}

\newtheorem{lemma}[theorem]{Lemma}

\newtheorem{proposition}[theorem]{Proposition}

\newenvironment{proof}[1][Proof]{\noindent\textbf{#1.} }{\ \rule{0.5em}{0.5em}}
\input{tcilatex}
\begin{document}

\title{\textbf{Incomplete Tribonacci-Lucas numbers and polynomials}}
\author{\textbf{Nazmiye Yilmaz }and\textbf{\ Necati Taskara} \\
Department of Mathematics, Science Faculty,\\
Selcuk University, Campus, 42075, Konya, Turkey\medskip \\
\textit{nzyilmaz@selcuk.edu.tr} and\ \textit{ntaskara@selcuk.edu.tr}}
\maketitle

\begin{abstract}
In this paper, we define Tribonacci-Lucas polynomials and present
Tribonacci-Lucas numbers and polynomials as a binomial sum. Then, we
introduce incomplete Tribonacci-Lucas numbers and polynomials. In addition
we derive recurrence relations, some properties and generating functions of
these numbers and polynomials. Also, we find the generating function of
incomplete Tribonacci polynomials which is given as the open problem in [12].


\end{abstract}

\section{Introduction}

\qquad Recently, Fibonacci and Lucas numbers have investigated very largely
and authors tried to developed and give some directions to mathematical
calculations using these type of special numbers [2,9,16,17]. One of these
directions goes through to the \textit{Tribonacci }and the \textit{%
Tribonacci-Lucas numbers. }In fact Tribonacci numbers have been firstly
defined by M. Feinberg in 1963 and then some important properties over this
numbers have been created by [5,8,10,13,19]. On the other hand,
Tribonacci-Lucas numbers have been introduced and investigated by author in
[4]. In addition, there exists another mathematical term, namely to be 
\textit{incomplete}, on Fibonacci, Lucas and Tribonacci numbers. As a brief
background, the incomplete Fibonacci, Lucas and Tribonacci numbers were
introduced by authours [6,11,12,18], and further the generating functions of
these numbers were presented by authours. Moreover, in [14,15], it is
defined and examined recurrence relations of the incomplete Fibonacci, Lucas 
$p$-numbers and $p$-polynomials . We may also refer [3,18] for further
studies about to have incompleteness of special numbers.

For $n\geq 2$, it is known that while the Tribonacci sequence $\left\{
T_{n}\right\} _{n\in 
\mathbb{N}
}$ is defined by 
\begin{equation}
T_{n+1}=T_{n}+T_{n-1}+T_{n-2}\ \ \ (T_{0}=0,\ T_{1}=T_{2}=1),  \label{1.1}
\end{equation}%
the Tribonacci-Lucas sequence $\left\{ K_{n}\right\} _{n\in 
\mathbb{N}
}$ is defined by

\begin{equation}
K_{n+1}=K_{n}+K_{n-1}+K_{n-2}\ \ (K_{0}=3,\ K_{1}=1,\ K_{2}=3).  \label{1.2}
\end{equation}%
There is also well known that each of the Tribonacci and Tribonaccci-Lucas
numbers is actually a linear combination of $\alpha ^{n}$, $\beta ^{n}$ and $%
\gamma ^{n}.$ In other words, 
\begin{equation}
\left. 
\begin{array}{c}
T_{n}=\frac{\alpha ^{n+1}}{(\alpha -\beta )(\alpha -\gamma )}+\frac{\beta
^{n+1}}{(\beta -\alpha )(\beta -\gamma )}+\frac{\gamma ^{n+1}}{(\gamma
-\alpha )(\gamma -\beta )} \\ 
\text{and} \\ 
K_{n}=\alpha ^{n}+\beta ^{n}+\gamma ^{n},%
\end{array}%
\right\}  \label{1.3}
\end{equation}%
where $\alpha ,\beta $ and $\gamma $ are roots of the characteristic
equations of (\ref{1.1}) and (\ref{1.2})\ such that 
\begin{eqnarray*}
\alpha &=&\frac{1+\sqrt[3]{19+3\sqrt{33}}+\sqrt[3]{19-3\sqrt{33}}}{3},\
\beta =\frac{1+w\sqrt[3]{19+3\sqrt{33}}+w^{2}\sqrt[3]{19-3\sqrt{33}}}{3}, \\
\ \gamma &=&\frac{1+w^{2}\sqrt[3]{19+3\sqrt{33}}+w\sqrt[3]{19-3\sqrt{33}}}{3}%
,\ \ w=\frac{-1+i\sqrt{3}}{2}.
\end{eqnarray*}%
Meanwhile we note that equations in (\ref{1.3}) are called the Binet
formulas for Tribonacci and Tribonacci-Lucas numbers, respectively.

Moreover, authors defined a large class \ of polynomials by Fibonacci and
Tribonacci numbers [1, 7]. Such polynomials are called Fibonacci polynomials
and Tribonacci polynomials \cite{9}, respectively. In 1973, Hoggatt and
Bicknell \cite{7} introduced Tribonacci polynomials. The Tribonacci
polynomials $T_{n}\left( x\right) $ are defined by the recurrence relation%
\begin{equation*}
T_{n+3}\left( x\right) =x^{2}T_{n+2}\left( x\right) +xT_{n+1}\left( x\right)
+T_{n}\ \left( x\right) ,
\end{equation*}%
where $T_{0}\left( x\right) =0,\ T_{1}\left( x\right) =1,\ T_{2}\left(
x\right) =x^{2}.$

On the other hand, in [12], incomplete Tribonacci polynomials are defined by%
\begin{equation}
T_{n}^{\left( s\right) }\left( x\right)
=\sum\limits_{i=0}^{s}\sum\limits_{j=0}^{i}\dbinom{i}{j}\dbinom{n-i-j-1}{i}%
x^{2n-3\left( i+j\right) -2},  \label{1.4}
\end{equation}%
where $0\leq s\leq \left\lfloor \frac{n-1}{2}\right\rfloor $. In here, for $%
x=1$, it is obtained incomplete Tribonacci numbers. The recurrence relation
of these polynomials is%
\begin{equation}
\begin{tabular}{ll}
$\ T_{n+3}^{\left( s\right) }\left( x\right) =$ & $x^{2}T_{n+2}^{\left(
s\right) }\left( x\right) +xT_{n+1}^{\left( s\right) }\left( x\right)
+T_{n}^{\left( s\right) }\left( x\right) -x\sum\limits_{j=0}^{s}\dbinom{s}{j}%
\dbinom{n-s-j}{s}x^{2n-3\left( s+j\right) }\ \ $ \\ 
& $\ -\sum\limits_{j=0}^{s}\dbinom{s}{j}\dbinom{n-s-j-1}{s}x^{2n-3\left(
s+j\right) -2}$%
\end{tabular}
\label{1.5}
\end{equation}%
$\ \ $

Also the generating function of incomplete Tribonacci numbers is%
\begin{equation}
Q_{s}\left( z\right) =\frac{T_{2s+1}+z\left( T_{2s+2}-T_{2s+1}\right)
+z^{2}\left( T_{2s+3}-T_{2s+2}-T_{2s+1}-2\right) -g\left( z\right) }{%
1-z-z^{2}-z^{3}},  \label{1.6}
\end{equation}%
where $g\left( z\right) =\left( z^{2}+z^{3}\right) \frac{\left( 1+z\right)
^{s}}{\left( 1-z\right) ^{s+1}}$ and $T_{n}$ is $n$-th Tribonacci number.

In the light of the above paragraph, the main goal of this paper is to
improve the Tribonacci-Lucas numbers\ with a different viewpoint . In order
to do that we first define Tribonacci-Lucas polynomials and then by
presenting Tribonacci-Lucas numbers and polynomials as a binomial sum\textit{%
, }we define\textit{\ }the incomplete Tribonacci-Lucas numbers and
polynomials.

After that we find the generating function of incomplete Tribonacci
polynomials which is given as the open problem in [12]. Also, we obtain some
properties and generating functions of incomplete Tribonacci-Lucas numbers
and polynomials.

\section{Tribonacci-Lucas polynomials and pascal-like triangle}

\qquad In the following table, we give the pascal-like triangle of
Tribonacci-Lucas numbers and each element of this table is defined in
similar way as in the tribonacci triangle.

\begin{eqnarray*}
&&%
\begin{tabular}[t]{l|lllllll}
$n\backslash i$ & $0$ & $1$ & $2$ & $3$ & $4$ & $5$ & $\cdots $ \\ \hline
$0$ & $3$ &  &  &  &  &  &  \\ 
$1$ & $1$ & $2$ &  &  &  &  &  \\ 
$2$ & $1$ & $6$ & $2$ &  &  &  &  \\ 
$3$ & $1$ & $8$ & $10$ & $2$ &  &  &  \\ 
$4$ & $1$ & $10$ & $24$ & $14$ & $2$ &  &  \\ 
$5$ & $1$ & $12$ & $42$ & $48$ & $18$ & $2$ &  \\ 
$\vdots $ &  &  &  & $\vdots $ &  &  & 
\end{tabular}
\\
&&\text{\textit{Table1}. Tribonacci-Lucas triangle}
\end{eqnarray*}

\qquad Let $B\left( n,i\right) $ be the element in the $n$-th row and $i$-th
column of the Tribonacci-Lucas triangle. By using the triangle, we have

\begin{equation}
B\left( n+1,i\right) =B\left( n,i\right) +B\left( n,i-1\right) +B\left(
n-1,i-1\right) ,  \label{2.1}
\end{equation}%
where $B\left( n,0\right) =1,\ B\left( n,n\right) =2$ for $n\in 
\mathbb{Z}
^{+}.$

By using the Table 1, we have the Tribonacci-Lucas numbers as binomial sum 
\begin{equation*}
K_{n}=\sum\limits_{i=0}^{\left\lfloor \frac{n}{2}\right\rfloor }B\left(
n-i,i\right) .
\end{equation*}%
In here, the sum of elements on the rising diagonal lines in the
Tribonacci-Lucas triangle is the Tribonacci-Lucas number $K_{n}$.
Furthermore, we write 
\begin{equation}
K_{n}=\sum\limits_{i=0}^{\left\lfloor \frac{n}{2}\right\rfloor
}\sum\limits_{j=0}^{i}\frac{n}{n-i-j}\dbinom{i}{j}\dbinom{n-i-j}{i},\ \ \ (\
n>i+j)  \label{2.2}
\end{equation}%
since these coefficients hold the relation 
\begin{equation*}
\left\{ 
\begin{array}{c}
B\left( n,i\right) =\sum\limits_{j=0}^{i}\frac{n+i}{n-j}\binom{i}{j}\binom{%
n-j}{i},\ \ \ (n>i) \\ 
B(n,i)=2,\ \ \ \ \ \ \ \ \ \ \ \ \ \ \ \ \ (n=i)%
\end{array}%
\right. .
\end{equation*}%
\qquad

Now, we introduce the Tribonacci-Lucas polynomial%
\begin{equation*}
K_{n+3}\left( x\right) =x^{2}K_{n+2}\left( x\right) +xK_{n+1}\left( x\right)
+K_{n}\ \left( x\right) ,\ \ 
\end{equation*}%
where $K_{0}\left( x\right) =3,\ K_{1}\left( x\right) =x^{2},\ K_{2}\left(
x\right) =x^{4}+2x.$ Note that $K_{n}\left( 1\right) =K_{n},$ $n\in 
\mathbb{N}
$. \ It is given a few Tribonacci-Lucas polynomials in the following:%
\begin{equation*}
\begin{tabular}{ll}
$K_{0}\left( x\right) =3,$ & $K_{4}\left( x\right) =x^{8}+4x^{5}+6x^{2},$ \\ 
$K_{1}\left( x\right) =x^{2},$ & $K_{5}\left( x\right)
=x^{10}+5x^{7}+10x^{4}+5x,$ \\ 
$K_{2}\left( x\right) =x^{4}+2x,$ & $K_{6}\left( x\right)
=x^{12}+6x^{9}+15x^{6}+14x^{3}+3,$ \\ 
$K_{3}\left( x\right) =x^{6}+3x^{3}+3,$ & $K_{7}\left( x\right)
=x^{14}+7x^{11}+21x^{8}+28x^{5}+14x^{2}.$%
\end{tabular}%
\end{equation*}%
\qquad \qquad \qquad\ \ \qquad

In similarly with the Tribonacci-Lucas triangle, we define the
Tribonacci-Lucas polynomials triangle:%
\begin{eqnarray*}
&&%
\begin{tabular}[t]{l|lllllll}
$n\backslash i$ & $0$ & $1$ & $2$ & $3$ & $4$ & $5$ & $\cdots $ \\ \hline
$0$ & $3$ &  &  &  &  &  &  \\ 
$1$ & $x^{2}$ & $2x$ &  &  &  &  &  \\ 
$2$ & $x^{4}$ & $3x^{3}+3$ & $2x^{2}$ &  &  &  &  \\ 
$3$ & $x^{6}$ & $4x^{5}+4x^{2}$ & $5x^{4}+5x$ & $2x^{3}$ &  &  &  \\ 
$4$ & $x^{8}$ & $5x^{7}+5x^{4}$ & $9x^{6}+12x^{3}+3$ & $7x^{5}+7x^{2}$ & $%
2x^{4}$ &  &  \\ 
$5$ & $x^{10}$ & $6x^{9}+6x^{6}$ & $14x^{8}+21x^{5}+7x^{2}$ & $%
16x^{7}+24x^{4}+8x$ & $9x^{6}+9x^{3}$ & $2x^{5}$ &  \\ 
$\vdots $ &  &  &  & $\vdots $ &  &  & 
\end{tabular}
\\
&&\text{\textit{Table2}. Tribonacci-Lucas polynomials triangle}
\end{eqnarray*}%
\qquad

Let $B\left( n,i\right) \left( x\right) $ be the element in the $n$-th row
and $i$-th column of the Tribonacci-Lucas polynomials triangle. By using the
triangle, we have

\begin{equation}
B\left( n+1,i\right) \left( x\right) =x^{2}B\left( n,i\right) \left(
x\right) +xB\left( n,i-1\right) \left( x\right) +B\left( n-1,i-1\right)
\left( x\right) ,  \label{2.3}
\end{equation}%
where $B\left( n,0\right) \left( x\right) =x^{2n},\ B\left( n,n\right)
\left( x\right) =2x^{n}$ for $n\in 
\mathbb{Z}
^{+}.$

By using the Table 2, we have 
\begin{equation*}
K_{n}\left( x\right) =\sum\limits_{i=0}^{\left\lfloor \frac{n}{2}%
\right\rfloor }B\left( n-i,i\right) \left( x\right) .
\end{equation*}%
In here, the sum of elements on the rising diagonal lines in the Table 2 is
the Tribonacci-Lucas polynomials $K_{n}\left( x\right) .$Furthermore, we
write 
\begin{equation}
K_{n}\left( x\right) =\sum\limits_{i=0}^{\left\lfloor \frac{n}{2}%
\right\rfloor }\sum\limits_{j=0}^{i}\frac{n}{n-i-j}\dbinom{i}{j}\dbinom{n-i-j%
}{i}x^{2n-3i-3j},\ \ \ \ \ (n>i+j)  \label{2.4}
\end{equation}%
since these coefficients satisfy the relation%
\begin{equation*}
\left\{ 
\begin{array}{c}
B\left( n,i\right) \left( x\right) =\sum\limits_{j=0}^{i}\frac{n+i}{n-j}%
\binom{i}{j}\binom{n-j}{i}x^{2n-i-3j},\ \ \ \ \ \ (n>i) \\ 
B(n,i)(x)=2x^{n},\ \ \ \ \ \ \ \ \ \ \ \ \ \ \ \ \ \ \ \ \ \ \ \ \ \ \ \ \ \
(n=i)%
\end{array}%
\right. .
\end{equation*}%
\qquad

Thus, by considering equations (\ref{2.2}) and (\ref{2.4}), we introduce the
incomplete Tribonacci-Lucas numbers and incomplete Tribonacci-Lucas
polynomials.

Now, we get new recurrence relations, some properties and generating
functions of incomplete Tribonacci-Lucas numbers and polynomials.

\section{The incomplete Tribonacci-Lucas polynomials and Tribonacci-Lucas
numbers}

\begin{definition}
The incomplete Tribonacci-Lucas polynomials $K_{n}^{\left( s\right) }\left(
x\right) $ are defined by%
\begin{eqnarray}
K_{n}^{\left( s\right) }\left( x\right) &=&\sum\limits_{i=0}^{s}B\left(
n-i,i\right) \left( x\right)  \label{3.1} \\
&=&\sum\limits_{i=0}^{s}\sum\limits_{j=0}^{i}\frac{n}{n-i-j}\dbinom{i}{j}%
\dbinom{n-i-j}{i}x^{2n-3\left( i+j\right) },  \notag
\end{eqnarray}%
where $0\leq s\leq \left\lfloor \frac{n}{2}\right\rfloor $ for $n\in 
\mathbb{Z}
^{+}.$
\end{definition}

In Definition 1, for $x=1$, we define the incomplete Tribonacci-Lucas
numbers, that is, $K_{n}^{\left( s\right) }\left( 1\right) =K_{n}\left(
s\right) .$

To reveal the importance of this subject, we can express the relationships
as in the following:

\begin{itemize}
\item $K_{n}^{\left( \left\lfloor \frac{n}{2}\right\rfloor \right) }\left(
x\right) =K_{n}\left( x\right) $ $\ $(the relationship between incomplete
Tribonacci-Lucas polynomials and Tribonacci-Lucas polynomials),

\item $K_{n}\left( \left\lfloor \frac{n}{2}\right\rfloor \right) =K_{n}$
(the relationship between incomplete Tribonacci-Lucas numbers and
Tribonacci-Lucas numbers).
\end{itemize}

From Definition 1, we have a few incomplete Tribonacci-Lucas polynomials as

\begin{eqnarray*}
&&%
\begin{tabular}[t]{l|lllll}
$n\backslash s$ & $0$ & $1$ & $2$ & $3$ & $\cdots $ \\ \hline
$1$ & $x^{2}$ &  &  &  &  \\ 
$2$ & $x^{4}$ & $x^{4}+2x$ &  &  &  \\ 
$3$ & $x^{6}$ & $x^{6}+3x^{3}+3$ &  &  &  \\ 
$4$ & $x^{8}$ & $x^{8}+4x^{5}+4x^{2}$ & $x^{8}+4x^{5}+6x^{2}$ &  &  \\ 
$5$ & $x^{10}$ & $x^{10}+5x^{7}+5x^{4}$ & $x^{10}+5x^{7}+10x^{4}+5x$ &  & 
\\ 
$6$ & $x^{12}$ & $x^{12}+6x^{9}+6x^{6}$ & $x^{12}+6x^{9}+15x^{6}+12x^{3}+3$
& $x^{12}+6x^{9}+15x^{6}+14x^{3}+3$ &  \\ 
$\vdots $ &  &  & $\vdots $ &  & 
\end{tabular}
\\
&&\text{\textit{Table3}. Incomplete Tribonacci-Lucas polynomials }
\end{eqnarray*}

By taking account of Table 3, we can write

\begin{equation}
K_{n}^{\left( 0\right) }\left( x\right) =x^{2n},  \label{3.3}
\end{equation}%
\begin{equation}
K_{n}^{\left( 1\right) }\left( x\right) =x^{2n}+nx^{2n-3}+nx^{2n-6},\ \
n\geq 3  \label{3.4}
\end{equation}

\begin{equation}
K_{n}^{\left( \left\lfloor \frac{n}{2}\right\rfloor \right) }\left( x\right)
=K_{n}\left( x\right) ,  \label{3.5}
\end{equation}

\begin{equation}
K_{n}^{\left( \left\lfloor \frac{n-2}{2}\right\rfloor \right) }\left(
x\right) =\left\{ 
\begin{array}{c}
\ \ \ K_{n}\left( x\right) -2x^{\frac{n}{2}},\ \ \ \ \ \ \ \ \ \ \ \ \ \ \ \
\ \ \ \ \ \ \ \ \ \ n\geq 2,\ even \\ 
K_{n}\left( x\right) -\left( nx^{\frac{n+3}{2}}+nx^{\frac{n-3}{2}}\right) ,\
\ \ \ \ \ \ \ n\geq 2,\ odd%
\end{array}%
.\right.  \label{3.6}
\end{equation}

\begin{proposition}
For $\ 0\leq s\leq \left\lfloor \frac{n}{2}\right\rfloor $ and $n\in 
\mathbb{Z}
^{+}$, we have the following recurrence relations;
\end{proposition}

\begin{enumerate}
\item[i)] \textit{The homogeneous recurrence relation of the incomplete
Tribonacci-Lucas polynomials }$K_{n}^{\left( s\right) }\left( x\right) $%
\textit{\ is}%
\begin{equation}
K_{n+3}^{\left( s+1\right) }\left( x\right) =x^{2}K_{n+2}^{\left( s+1\right)
}\left( x\right) +xK_{n+1}^{\left( s\right) }\left( x\right) +K_{n}^{\left(
s\right) }\left( x\right) .  \label{3.7}
\end{equation}

\item[ii)] \textit{The non-homogeneous recurrence relation of the incomplete
Tribonacci-Lucas polynomials }$K_{n}^{\left( s\right) }\left( x\right) $%
\textit{\ is}%
\begin{eqnarray}
K_{n+3}^{\left( s\right) }\left( x\right) &=&x^{2}K_{n+2}^{\left( s\right)
}\left( x\right) +xK_{n+1}^{\left( s\right) }\left( x\right) +K_{n}^{\left(
s\right) }\left( x\right)  \label{3.8} \\
&&-xB\left( n+1-s,s\right) \left( x\right) -B\left( n-s,s\right) \left(
x\right) .  \notag
\end{eqnarray}
\end{enumerate}

\begin{proof}

\begin{description}
\item[i)] From Definition 1, let us label $x^{2}K_{n+2}^{\left( s+1\right)
}\left( x\right) +xK_{n+1}^{\left( s\right) }\left( x\right) +K_{n}^{\left(
s\right) }\left( x\right) $ by $RHS$. Actually we write%
\begin{eqnarray*}
RHS &=&x^{2}\sum\limits_{i=0}^{s+1}B\left( n+2-i,i\right) \left( x\right)
+x\sum\limits_{i=0}^{s}B\left( n+1-i,i\right) \left( x\right)
+\sum\limits_{i=0}^{s}B\left( n-i,i\right) \left( x\right) \\
&=&x^{2}\sum\limits_{i=0}^{s+1}B\left( n+2-i,i\right) \left( x\right)
+x\sum\limits_{i=1}^{s+1}B\left( n+2-i,i-1\right) \left( x\right)
+\sum\limits_{i=1}^{s+1}B\left( n+1-i,i-1\right) \left( x\right) \\
&=&\sum\limits_{i=0}^{s+1}\left( x^{2}B\left( n+2-i,i\right) \left( x\right)
+xB\left( n+2-i,i-1\right) \left( x\right) +B\left( n+1-i,i-1\right) \left(
x\right) \right) \\
&&-xB\left( n+2,-1\right) \left( x\right) -B\left( n+1,-1\right) \left(
x\right) .
\end{eqnarray*}%
Then, by considering $\dbinom{n}{-1}=0$ and the equation (\ref{2.3}), we
finally have%
\begin{equation*}
RHS=\sum\limits_{i=0}^{s+1}B\left( n+3-i,i\right) \left( x\right)
=K_{n+3}^{\left( s+1\right) }\left( x\right)
\end{equation*}%
as required.

\item[ii)] By considering the equations (\ref{2.3}), (\ref{3.7}) and
Definition 1, we have%
\begin{eqnarray*}
\sum\limits_{i=0}^{s+1}B\left( n+3-i,i\right) \left( x\right)
&=&x^{2}\sum\limits_{i=0}^{s+1}B\left( n+2-i,i\right) \left( x\right)
+x\sum\limits_{i=0}^{s}B\left( n+1-i,i\right) \left( x\right)
+\sum\limits_{i=0}^{s}B\left( n-i,i\right) \left( x\right) \\
\sum\limits_{i=0}^{s}B\left( n+3-i,i\right) \left( x\right)
&=&x^{2}\sum\limits_{i=0}^{s}B\left( n+2-i,i\right) \left( x\right)
+x\sum\limits_{i=0}^{s}B\left( n+1-i,i\right) \left( x\right)
+\sum\limits_{i=0}^{s}B\left( n-i,i\right) \left( x\right) \\
&&-B\left( n+2-s,s+1\right) +x^{2}B\left( n+1-s,s+1\right) \\
K_{n+3}^{\left( s\right) }\left( x\right) &=&x^{2}K_{n+2}^{\left( s\right)
}\left( x\right) +xK_{n+1}^{\left( s\right) }\left( x\right) +K_{n}^{\left(
s\right) }\left( x\right) -xB\left( n+1-s,s\right) -B\left( n-s,s\right) .
\end{eqnarray*}
\end{description}
\end{proof}

By using Table 3, for $x=1$, we have incomplete Tribonacci-Lucas numbers in
the following Table 4:%
\begin{eqnarray*}
&&%
\begin{tabular}[t]{l|lllll}
$n\backslash s$ & $0$ & $1$ & $2$ & $3$ & $\cdots $ \\ \hline
$1$ & $1$ &  &  &  &  \\ 
$2$ & $1$ & $3$ &  &  &  \\ 
$3$ & $1$ & $7$ &  &  &  \\ 
$4$ & $1$ & $9$ & $11$ &  &  \\ 
$5$ & $1$ & $11$ & $21$ &  &  \\ 
$6$ & $1$ & $13$ & $37$ & $39$ &  \\ 
$\vdots $ &  &  & $\vdots $ &  & 
\end{tabular}
\\
&&\text{\textit{Table4}. Incomplete Tribonacci-Lucas numbers }
\end{eqnarray*}

\begin{corollary}
For $\ 0\leq s\leq \left\lfloor \frac{n}{2}\right\rfloor $ and $n\in 
\mathbb{Z}
^{+},$ we have the following recurrence relations;
\end{corollary}

\begin{enumerate}
\item[i)] \textit{The homogeneous recurrence relation of the incomplete
Tribonacci-Lucas numbers }$K_{n}\left( s\right) $\textit{\ is}%
\begin{equation}
K_{n+3}\left( s+1\right) =K_{n+2}\left( s+1\right) +K_{n+1}\left( s\right)
+K_{n}\left( s\right) .  \label{3.9}
\end{equation}

\item[ii)] \textit{The non-homogeneous recurrence relation of the incomplete
Tribonacci-Lucas numbers }$K_{n}\left( s\right) $\textit{\ is}%
\begin{eqnarray}
K_{n+3}\left( s\right) &=&K_{n+2}\left( s\right) +K_{n+1}\left( s\right)
+K_{n}\left( s\right)  \label{3.10} \\
&&-B\left( n+1-s,s\right) -B\left( n-s,s\right) .  \notag
\end{eqnarray}
\end{enumerate}

\begin{proposition}
The relation between of incomplete Tribonacci polynomials $T_{n}^{\left(
s\right) }\left( x\right) $ and incomplete Tribonacci-Lucas polynomials $%
K_{n}^{\left( s\right) }\left( x\right) $ is%
\begin{equation*}
K_{n}^{\left( s\right) }\left( x\right) =T_{n+1}^{\left( s\right) }\left(
x\right) +xT_{n-1}^{\left( s-1\right) }\left( x\right) +2T_{n-2}^{\left(
s-1\right) }\left( x\right) ,
\end{equation*}%
where $1\leq s\leq \left\lfloor \frac{n-1}{2}\right\rfloor $ and $n>2.$
\end{proposition}

\begin{proof}
Proof of its can easily do by using Definition 1 and the equation (\ref{1.4}%
).
\end{proof}

\begin{corollary}
The relation between of incomplete Tribonacci numbers $T_{n}\left( s\right) $
and incomplete Tribonacci-Lucas numbers $K_{n}\left( s\right) $ is%
\begin{equation*}
K_{n}\left( s\right) =T_{n+1}\left( s\right) +T_{n-1}\left( s-1\right)
+2T_{n-2}\left( s-1\right) ,
\end{equation*}%
where $1\leq s\leq \left\lfloor \frac{n-1}{2}\right\rfloor $ and $n>2.$
\end{corollary}

\begin{theorem}
For $n,\ h\geq 1$ and $0\leq s\leq \left\lfloor \frac{n}{2}\right\rfloor ,$
the sum of incomplete Tribonacci-Lucas numbers is%
\begin{equation}
\sum_{i=0}^{h-1}K_{n+i}\left( s\right) =\frac{1}{2}\left( K_{n+h+2}\left(
s+1\right) -K_{n+2}\left( s+1\right) +K_{n}\left( s\right) -K_{n+h}\left(
s\right) \right) .  \label{3.11}
\end{equation}
\end{theorem}

\begin{proof}
Let us use the principle of mathematical induction on $h$ to prove (\ref%
{3.11}). While, for $h=1$, it is easy to see that%
\begin{equation*}
K_{n}\left( s\right) =\frac{1}{2}\left( K_{n+3}\left( s+1\right)
-K_{n+2}\left( s+1\right) +K_{n}\left( s\right) -K_{n+1}\left( s\right)
\right) .
\end{equation*}%
As the usual next step of inductions, let us assume that it is true for all
positive integers $h.$ That is,%
\begin{equation*}
\sum_{i=0}^{h-1}K_{n+i}\left( s\right) =\frac{1}{2}\left( K_{n+h+2}\left(
s+1\right) -K_{n+2}\left( s+1\right) +K_{n}\left( s\right) -K_{n+h}\left(
s\right) \right) .
\end{equation*}%
Therefore, we have to show that it is true for $h+1.$ In other words$,$ we
need to check%
\begin{equation*}
\sum_{i=0}^{h}K_{n+i}\left( s\right) =\frac{1}{2}\left( K_{n+h+3}\left(
s+1\right) -K_{n+2}\left( s+1\right) +K_{n}\left( s\right) -K_{n+h+1}\left(
s\right) \right) .
\end{equation*}%
Hence, we can write%
\begin{eqnarray*}
\sum_{i=0}^{h}K_{n+i}\left( s\right) &=&\sum_{i=0}^{h-1}K_{n+i}\left(
s\right) +K_{n+h}\left( s\right) \\
&=&\frac{1}{2}\left( K_{n+h+2}\left( s+1\right) -K_{n+2}\left( s+1\right)
+K_{n}\left( s\right) -K_{n+h}\left( s\right) \right) +K_{n+h}\left( s\right)
\\
&=&\frac{1}{2}\left( K_{n+h+2}\left( s+1\right) -K_{n+2}\left( s+1\right)
+K_{n}\left( s\right) +K_{n+h}\left( s\right) \right) \\
&=&\frac{1}{2}\left( K_{n+h+3}\left( s+1\right) -K_{n+2}\left( s+1\right)
+K_{n}\left( s\right) -K_{n+h+1}\left( s\right) \right) .
\end{eqnarray*}
\end{proof}

The following proposition give the sum of incomplete Tribonacci-Lucas
poynomials, that is, sum of the $n$-th row of the Table 3.

\begin{proposition}
For $l=\left\lfloor \frac{n}{2}\right\rfloor ,$ we have the equality%
\begin{equation}
\sum_{s=0}^{l}K_{n}^{\left( s\right) }\left( x\right) =\left( l+1\right)
K_{n}\left( x\right) -\sum\limits_{i=0}^{l}\sum\limits_{j=0}^{i}\frac{in}{%
n-i-j}\dbinom{i}{j}\dbinom{n-i-j}{i}x^{2n-3\left( i+j\right) }.  \label{3.12}
\end{equation}
\end{proposition}

\begin{proof}
From Definition 1, we have%
\begin{equation*}
K_{n}^{\left( s\right) }\left( x\right) =\sum\limits_{i=0}^{s}B\left(
n-i,i\right) \left( x\right) .
\end{equation*}%
Then, we can write%
\begin{eqnarray*}
\sum_{s=0}^{l}K_{n}^{\left( s\right) }\left( x\right)
&=&\sum\limits_{s=0}^{l}\sum\limits_{i=0}^{s}B\left( n-i,i\right) \left(
x\right) \\
&=&\left( l+1\right) B\left( n,0\right) \left( x\right) +lB\left(
n-1,1\right) \left( x\right) +\cdots +B\left( n-l,l\right) \left( x\right) \\
&=&\sum\limits_{i=0}^{l}\left( l+1-i\right) B\left( n-i,i\right) \left(
x\right) \\
&=&\sum\limits_{i=0}^{l}\left( l+1\right) B\left( n-i,i\right) \left(
x\right) -\sum\limits_{i=0}^{l}iB\left( n-i,i\right) \left( x\right) \\
&=&\left( l+1\right) K_{n}^{\left( l\right) }\left( x\right)
-\sum\limits_{i=0}^{l}\sum\limits_{j=0}^{i}\frac{in}{n-i-j}\dbinom{i}{j}%
\dbinom{n-i-j}{i}x^{2n-3\left( i+j\right) }.
\end{eqnarray*}%
Since $l$ is $\left\lfloor \frac{n}{2}\right\rfloor ,$ we obtain%
\begin{equation*}
\sum_{s=0}^{l}K_{n}^{\left( s\right) }\left( x\right) =\left( l+1\right)
K_{n}\left( x\right) -\sum\limits_{i=0}^{l}\sum\limits_{j=0}^{i}\frac{in}{%
n-i-j}\dbinom{i}{j}\dbinom{n-i-j}{i}x^{2n-3\left( i+j\right) }.
\end{equation*}
\end{proof}

The following corollary shows the sum of the $n$-th row of the Table 4. It
is obtained from (\ref{3.12}) with $x=1.$

\begin{corollary}
For $l=\left\lfloor \frac{n}{2}\right\rfloor ,$ we have the following
equality;%
\begin{equation}
\sum_{s=0}^{l}K_{n}\left( s\right) =\left( l+1\right)
K_{n}-\sum\limits_{i=0}^{l}\sum\limits_{j=0}^{i}\frac{in}{n-i-j}\dbinom{i}{j}%
\dbinom{n-i-j}{i}.  \label{3.13}
\end{equation}
\end{corollary}

\section{Generating functions of the incomplete Tribonacci and
Tribonacci-Lucas polynomials and numbers}

\begin{lemma}
Let $\left\{ S_{n}\right\} _{n=0}^{\infty }$ be a complex sequence
satisfying the non-homogeneous third-order recurrence relation. Then we have%
\begin{equation*}
S_{n}=aS_{n-1}+bS_{n-2}+cS_{n-3}+r_{n},
\end{equation*}%
where $a,b,c\in 
\mathbb{C}
,~n\geq 3~$and $r_{n}:%
\mathbb{N}
\rightarrow 
\mathbb{C}
$ is a sequence. Hence the generating function $U\left( x\right) $ of $S_{n}$
is given by%
\begin{equation}
U\left( x\right) =\frac{S_{0}-r_{0}+x\left( S_{1}-aS_{0}-r_{1}\right)
+x^{2}\left( S_{2}-aS_{1}-bS_{0}-r_{2}\right) +G\left( x\right) }{%
1-ax-bx^{2}-cx^{3}},  \label{4.1}
\end{equation}%
where $G\left( x\right) $ denotes the generating function of $r_{n}$.
\end{lemma}

\begin{proof}
Let $U(x)$ and $G(x)$ be two generating functions for complex sequences $%
S_{n}$ and $r_{n}$, respectively, where%
\begin{equation}
U(x)=S_{0}+S_{1}x+S_{2}x^{2}+S_{3}x^{3}+\ldots +S_{n}x^{n}+\ldots ,
\label{4.2}
\end{equation}%
\begin{equation}
G\left( x\right) =r_{0}+r_{1}x+r_{2}x^{2}+r_{3}x^{3}+\ldots
+r_{n}x^{n}+\ldots .  \label{4.3}
\end{equation}%
If $U(x)$ given in (\ref{4.2}) multiply with $ax,$ $bx^{2}$ and $cx^{3},$
respectively, then we get%
\begin{equation}
\left. 
\begin{array}{c}
axU\left( x\right) =aS_{0}x+aS_{1}x^{2}+aS_{2}x^{3}+aS_{3}x^{4}+\ldots
+aS_{n}x^{n+1}+\ldots \\ 
bx^{2}U\left( x\right)
=bS_{0}x^{2}+bS_{1}x^{3}+bS_{2}x^{4}+bS_{3}x^{5}+\ldots +bS_{n}x^{n+2}+\ldots
\\ 
cx^{3}U\left( x\right)
=cS_{0}x^{3}+cS_{1}x^{4}+cS_{2}x^{5}+cS_{3}x^{6}+\ldots +cS_{n}x^{n+3}+\ldots%
\end{array}%
\right\} .  \label{4.4}
\end{equation}%
Consequently, by subtracting (\ref{4.3}) and (\ref{4.4}) from (\ref{4.2}),
it is obtained the equation%
\begin{equation*}
U\left( x\right) =\frac{S_{0}-r_{0}+x\left( S_{1}-aS_{0}-r_{1}\right)
+x^{2}\left( S_{2}-aS_{1}-bS_{0}-r_{2}\right) +G\left( x\right) }{%
1-ax-bx^{2}-cx^{3}}
\end{equation*}%
which completes the proof of the Lemma.\medskip
\end{proof}

Now, we examine the problem which is given for incomplete Tribonacci
polynomials in [15] .

\begin{theorem}
The generating function of the incomplete Tribonacci polynomials $%
T_{n}^{\left( s\right) }\left( x\right) $ is given by%
\begin{equation*}
Q_{s}\left( x,z\right) =\sum_{i=0}^{\infty }T_{i}^{\left( s\right) }\left(
x\right) z^{i}=z^{2s+1}U_{s}\left( x,z\right) ,
\end{equation*}%
where $U_{s}\left( x,z\right) =\frac{T_{2s+1}\left( x\right) +z\left(
T_{2s+2}\left( x\right) -x^{2}T_{2s+1}\left( x\right) \right) +z^{2}\left(
T_{2s}\left( x\right) -2x^{s+1}\right) -\left( xz^{2}+z^{3}\right) \frac{%
\left( x+z\right) ^{s}}{\left( 1-x^{2}z\right) ^{s+1}}}{\left(
1-x^{2}z-xz^{2}-z^{3}\right) }$.
\end{theorem}

\begin{proof}
Let $s$ be fixed positive integer. By using the equations (\ref{1.4}) and (%
\ref{1.5}), we have%
\begin{eqnarray*}
T_{n}^{\left( s\right) }\left( x\right) &=&0\ \ \ \ \ \ \left( 0\leq
n<2s+1\right) , \\
T_{2s+1}^{\left( s\right) }\left( x\right) &=&T_{2s+1}\left( x\right) , \\
T_{2s+2}^{\left( s\right) }\left( x\right) &=&T_{2s+2}\left( x\right) , \\
T_{2s+3}^{\left( s\right) }\left( x\right) &=&T_{2s+3}\left( x\right)
-x^{s+1},
\end{eqnarray*}%
and 
\begin{eqnarray}
T_{n}^{\left( s\right) }\left( x\right) &=&x^{2}T_{n-1}^{\left( s\right)
}\left( x\right) +xT_{n-2}^{\left( s\right) }\left( x\right)
+T_{n-3}^{\left( s\right) }\left( x\right) -\sum\limits_{j=0}^{s}\dbinom{s}{j%
}\dbinom{n-3-s-j}{s}x^{2n-5-3\left( s+j\right) }  \label{4.5} \\
&&-\sum\limits_{j=0}^{s}\dbinom{s}{j}\dbinom{n-4-s-j}{s}x^{2n-8-3\left(
s+j\right) },  \notag
\end{eqnarray}%
where $\ n\geq 4+2s.$ Also, we replace $S_{0},S_{1},...,S_{n}$ by $%
T_{2s+1}^{\left( s\right) }\left( x\right) ,~T_{2s+2}^{\left( s\right)
}\left( x\right) ,...,T_{n+2s+1}^{\left( s\right) }\left( x\right) $,
respectively. Assume that $r_{0}=r_{1}=0,\ r_{2}=x^{s+1}$ and 
\begin{equation*}
r_{n}=\sum\limits_{j=0}^{s}\dbinom{s}{j}\dbinom{n-2+s-j}{s}%
x^{2n-3+s-3j}+\sum\limits_{j=0}^{s}\dbinom{s}{j}\dbinom{n-3+s-j}{s}%
x^{2n-6+s-3j}.
\end{equation*}%
Furthermore, by considering [20, page 127]$,$ the generating function $%
G\left( x,z\right) $ of the $\left\{ r_{n}\right\} \ $is 
\begin{equation*}
G\left( x,z\right) =\left( xz^{2}+z^{3}\right) \frac{\left( x+z\right) ^{s}}{%
\left( 1-x^{2}z\right) ^{s+1}}.
\end{equation*}%
Therefore, by using Lemma 9, the generating function $U_{s}\left( x,z\right) 
$ of the sequence $S_{n}$ is 
\begin{eqnarray*}
U_{s}\left( x,z\right) \left( 1-x^{2}z-xz^{2}-z^{3}\right) +G\left(
x,z\right) &=&T_{2s+1}\left( x\right) +z\left( T_{2s+2}\left( x\right)
-x^{2}T_{2s+1}\left( x\right) \right) \\
&&+z^{2}\left( T_{2s}\left( x\right) -2x^{s+1}\right) .
\end{eqnarray*}%
Eventually, we conclude that $Q_{s}\left( x,z\right) =z^{2s+1}U_{s}\left(
x,z\right) $ as required.
\end{proof}

In the above theorem, if we take $x=1$, it is obtained the generating
function of the Tribonacci numbers in [12].

\begin{corollary}
The generating function of the Tribonacci numbers $T_{n}\left( s\right) $ is
given by%
\begin{equation*}
Q_{s}\left( z\right) =\sum_{i=0}^{\infty }T_{i}\left( s\right)
z^{i}=z^{2s+1}U_{s}\left( z\right) ,
\end{equation*}%
where $U_{s}\left( z\right) =\frac{T_{2s+1}+z\left( T_{2s+2}-T_{2s+1}\right)
+z^{2}\left( T_{2s}-2\right) -\left( z^{2}+z^{3}\right) \frac{\left(
1+z\right) ^{s}}{\left( 1-z\right) ^{s+1}}}{\left( 1-z-z^{2}-z^{3}\right) }$.
\end{corollary}

\begin{theorem}
The generating function of the incomplete Tribonacci-Lucas polynomials $%
K_{n}^{\left( s\right) }\left( x\right) $ is given by%
\begin{equation*}
W_{s}\left( x,z\right) =\sum_{n=0}^{\infty }K_{n}^{\left( s\right) }\left(
x\right) z^{n}=z^{-1}Q_{s}\left( x,z\right) +\left( xz+2z^{2}\right)
Q_{s-1}\left( x,z\right) ,
\end{equation*}%
where $s>1$ and $Q_{s}\left( x,z\right) $ is the generating function of
incomplete Tribonacci polynomials$.$
\end{theorem}

\begin{proof}
Let $W_{s}\left( x,z\right) $ be generating function of the incomplete
Tribonacci-Lucas polynomials, that is $W_{s}\left( x,z\right)
=\sum_{n=0}^{\infty }K_{n}^{\left( s\right) }\left( x\right) z^{n}.$

By using Proposition 4, Theorem 10 and the property of sum , we definitely
have%
\begin{eqnarray*}
\sum_{n=0}^{\infty }K_{n}^{\left( s\right) }\left( x\right) z^{n}
&=&\sum_{n=0}^{\infty }\left( T_{n+1}^{\left( s\right) }\left( x\right)
+xT_{n-1}^{\left( s-1\right) }\left( x\right) +2T_{n-2}^{\left( s-1\right)
}\left( x\right) \right) z^{n} \\
&=&z^{-1}Q_{s}\left( x,z\right) +\left( xz+2z^{2}\right) Q_{s-1}\left(
x,z\right) .
\end{eqnarray*}
\end{proof}

For $x=1$ in Theorem 12, we can present the generating function of
incomplete Tribonacci-Lucas numbers.

\begin{corollary}
The generating function of the incomplete Tribonacci-Lucas numbers $%
K_{n}\left( s\right) $ is given by%
\begin{equation*}
W_{s}\left( z\right) =\sum_{n=0}^{\infty }K_{n}\left( s\right)
z^{n}=z^{-1}Q_{s}\left( z\right) +\left( z+2z^{2}\right) Q_{s-1}\left(
z\right) ,
\end{equation*}%
where $s>1$ and $Q_{s}\left( z\right) $ is the generating function of
incomplete Tribonacci numbers$.$
\end{corollary}

\end{document}